\documentclass[11pt]{article}
\usepackage[left=1in,top=1in,right=1in,bottom=1in,head=0in,nofoot]{geometry}

\usepackage{setspace,url,bm,amsmath} 
\usepackage{scalerel}

\usepackage{xcolor}
 \usepackage{multirow}
 \usepackage{arydshln}
 \usepackage{mathtools}
\usepackage{titlesec} 
\titlelabel{\thetitle.\quad} 
\titleformat*{\section}{\bf\Large\center}

\usepackage{float}
\usepackage{graphicx} 
\usepackage{bbm}
\usepackage{latexsym}
\usepackage{caption}
\usepackage[margin=20pt]{subcaption}
\usepackage{enumerate}
\usepackage{tikz}
\usetikzlibrary{calc}

\usetikzlibrary{positioning}
\usepackage{amsthm}
\usepackage{amssymb}
\usepackage{amsmath}
\usepackage{color}

\usepackage[labelformat=simple]{subcaption}

\usepackage{booktabs,caption}
\usepackage{enumitem}
\usepackage{amsfonts}
\usepackage[all,import]{xy}
\def\T{{ \mathrm{\scriptscriptstyle T} }}
 
\newtheorem{definition}{Definition}
\newtheorem{lemma}{Lemma}

\newtheorem{theorem}{Theorem}
\newtheorem{example}{Example}
\newtheorem{proposition}{Proposition}

\usepackage{natbib} 
\bibpunct{(}{)}{;}{a}{}{,} 

\renewcommand{\refname}{References} 
\usepackage{etoolbox} 
\apptocmd{\sloppy}{\hbadness 10000\relax}{}{} 

\usepackage{multibib}
\newcites{sec}{References}

\usepackage{color}
\usepackage{listings}
\usepackage{hyperref}
\usepackage{booktabs}

\usepackage{lscape}

\RequirePackage[normalem]{ulem}

\allowdisplaybreaks

\begin{document}
\onehalfspacing

\title{\bf  
A stronger Sylvester’s criterion for positive semidefinite matrices
} 
\author{Mingrui Zhang and Peng Ding
\footnote{Mingrui Zhang, Division of Biostatistics, University of California, Berkeley, CA 94720 U.S.A. (E-mail: mingrui\_zhang@berkeley.edu). Peng Ding, Department of Statistics, University of California, Berkeley, CA 94720 U.S.A. (E-mail: pengdingpku@berkeley.edu). Ding was partially funded by the U.S. National Science Foundation (grant \# 1945136). 
}
}
\date{}
 
\maketitle

\begin{abstract}
Sylvester's criterion characterizes positive definite (PD) and positive semidefinite (PSD) matrices without the need of eigendecomposition. It states that a symmetric matrix is PD if and only if all of its leading principal minors are positive, and a symmetric matrix is PSD if and only if all of its principal minors are nonnegative. For an $m\times m$ symmetric matrix, Sylvester's criterion requires computing $m$ and $2^m-1$ determinants to verify it is PD and PSD, respectively. Therefore, it is less useful for PSD matrices due to the exponential growth in the number of principal submatrices as the matrix dimension increases. We provide a stronger Sylvester's criterion for PSD matrices which only requires to verify the nonnegativity of $m(m+1)/2$ determinants. Based on the new criterion, we provide a method to derive elementwise criteria for PD and PSD matrices. We illustrate the applications of our results in PD or PSD matrix completion and highlight their statistics applications via nonlinear semidefinite program. 
\end{abstract}

\medskip 
\noindent 
{\bf Keywords}: matrix completion; optimization; positive definite; positive semidefinite

\newpage

\section{Classic Sylvester's criterion and our new result}\label{sec1}

\subsection{Review of Sylvester's criterion}

Positive definite (PD) and positive semidefinite (PSD) matrices are important in mathematics and related fields. In this paper, we only discuss real matrices. For an $m\times m$ symmetric matrix $X$, it is PD is $c^{\T} X c >0$ for all non-zero $m$-dimensional vector $c$, and it is PSD if  $c^{\T} X c \geq 0$ for all $m$-dimensional vector $c$. There are different characterizations of PD and PSD matrices. The most common characterization is based on the eigenvalues, which requires the eigendecomposition of matrices. That is, a symmetric matrix is PD if all its eigenvalues are positive, and PSD if all its eigenvalues are nonnegative. Without the need of eigendecomposition, Sylvester's criterion determines whether or not a matrix is PD or PSD based on the determinants of its submatrices. Before reviewing Sylvester's criterion, we introduce some notation and definitions. 

For integers $a\leq b$, we use $a:b$ to denote the set of integers $\{a,a+1,...,b-1,b\}$. For any $m\times n$ matrix $X$, if $I$ is a sequence of unique values in $1:m$ and $J$ is a sequence of unique values in $1:n$, then we use $X_{I,J}$ to denote the submatrix of $X$ having the rows with indices in $I$ and the columns with indices in $J$. We use $\mathrm{det}(\cdot)$ to denote the determinant of a matrix. We review the definition of principal submatrix, principal minor, leading principal submatrix, and leading principal minors in Definition \ref{def::principal} below.

\begin{definition}\label{def::principal}
Consider an $m\times m$ matrix $X$. If $I$ is a subset of $\{1,...,m\}$, then $X_{I,I}$ is a principal submatrix of $X$ and $\mathrm{det}(X_{I,I})$ is a principal minor of $X$. If $b\geq 1$ is an integer, then $X_{1:b,1:b}$ is a leading principal submatrix of $X$ and $\mathrm{det}(X_{1:b,1:b})$ is a leading principal minor of $X$. 
\end{definition}

Based on Definition \ref{def::principal}, Sylvester's criterion gives sufficient and necessary conditions for PD and PSD matrices \citep{horn2012matrix}, which is reviewed in Theorem \ref{thm::syl} below. 

\begin{theorem}[Sylvester's criterion]\label{thm::syl}
(i) A symmetric matrix is PD if and only if all of its leading principal minors are positive. (ii) A symmetric matrix is PSD if and only if all of its principal minors are nonnegative.
\end{theorem}

By Theorem \ref{thm::syl}(i), to verify whether or not a symmetric $m\times m$ matrix $X$ is PD, we need to check the determinants of its $m$ leading principal submatrices. However, by Theorem \ref{thm::syl}(ii), to verify whether or not a symmetric $m\times m$ matrix $X$ is PSD, we need to check the determinants of its $2^m-1$ principal submatrices. Thus, Sylvester's criterion for PSD matrices is less useful due to the exponential growth in the number of principal submatrices as the matrix dimension increases. 

Theorem \ref{thm::syl}(i) is attributed to James Joseph Sylvester (1814–1897). However, Sylvester's criterion for PD matrices has been extended, without proper justification, to test PSD matrices; see \citet{kwan2021remedies} for a survey of such unjustified applications and extensions of Sylvester's criterion. Theorem \ref{thm::syl}(ii) was first introduced by \citet{prussing1986principal}, providing a correct and necessary condition for PSD matrices. 

\subsection{Statement of the new result}
Our new result improves Sylvester's criterion for PSD matrices by only checking the determinants of its $m(m+1)/2$ principal submatrices. We introduce some further definitions below. 

\begin{definition}\label{def::consecutive-principal}
For an $m\times m$ matrix $X$ and integers $a\leq b$, we call $X_{a:b,a:b}$ a consecutive principal submatrix of $X$. 
\end{definition}

By Definitions \ref{def::principal} and \ref{def::consecutive-principal}, a consecutive principal submatrix of $X$ must be a principal submatrix of $X$, and a leading principal submatrix of $X$ must be a consecutive principal submatrix of $X$. 

\begin{definition}\label{def::inner-saturated}
For a symmetric $m\times m$ matrix $X$, we define $X_{I,I}$ as an inner-saturated submatrix of $X$, where $I=\{1,m\}\cup J$ and the index set $J$ satisfies 
\begin{enumerate}
\item[(i)] when $m\leq 2$, $J=\emptyset$; 
\item[(ii)] when $m\geq 3$, $\{X_{2:(m-1),j}:j\in J\}$ is a maximal linearly independent set of the column vectors in $X_{2:(m-1),2:(m-1)}$. 
\end{enumerate}
\end{definition}

When $m=3$, the unique inner-saturated submatrix of $X$ is $X$ itself if $X_{2,2}\not=0$, and is $X_{\{1,3\},\{1,3\}}$ if $X_{2,2}=0$. This is because the maximal linearly independent set of $\{0\}$ is an empty set. When $m\geq 4$, we provide two examples below to illustrate the definition of the inner-saturated submatrix. 

\begin{example}
Consider the following $4\times 4$ symmetric matrix: 
\begin{equation*}
X=\begin{pmatrix}
1 & 2 & 3 & 4 \\
2 & 3 & 4 & 5 \\
3 & 4 & 5 & 6 \\
4 & 5 & 6 & 7 \\
\end{pmatrix}.
\end{equation*}
It is the unique inner-saturated submatrix of itself, because the column vectors of the matrix $\begin{pmatrix}
3 & 4\\
4 & 5
\end{pmatrix}$
are linearly independent. 
\end{example}

\begin{example}
Consider the following $5\times 5$ symmetric matrix: 
\begin{equation*}
X=\begin{pmatrix}
1 & 2 & 3 & 4 & 5 \\
2 & 3 & 4 & 5 & 6\\
3 & 4 & 5 & 6 & 7\\
4 & 5 & 6 & 7 & 8\\
5 & 6 & 7 & 8 & 9\\
\end{pmatrix}.
\end{equation*}
It has three inner-saturated submatrices removing 4th, 3rd, 2nd row and column from $X$, respectively
\begin{equation*}
Y_1=\begin{pmatrix}
1 & 2 & 3 & 5 \\
2 & 3 & 4 & 6\\
3 & 4 & 5 & 7\\
5 & 6 & 7 & 9\\
\end{pmatrix},\quad Y_2=\begin{pmatrix}
1 & 2 & 4 & 5 \\
2 & 3 & 5 & 6\\
4 & 5 & 7 & 8\\
5 & 6 & 8 & 9\\
\end{pmatrix},\quad Y_3=\begin{pmatrix}
1 & 3 & 4 & 5 \\
3 & 5 & 6 & 7\\
4 & 6 & 7 & 8\\
5 & 7 & 8 & 9\\
\end{pmatrix}
\end{equation*}
because the column vectors of the matrix $\begin{pmatrix}
3 & 4 & 5\\
4 & 5 & 6\\
5 & 6 & 7
\end{pmatrix}$
are linearly dependent, but the column vectors of the matrix $\begin{pmatrix}
3 & 4\\
4 & 5
\end{pmatrix}$, $\begin{pmatrix}
3 & 5\\
5 & 7
\end{pmatrix}$, $\begin{pmatrix}
5 & 6\\
6 & 7
\end{pmatrix}$
are linearly independent. 
\end{example}

We present our new result in Theorem \ref{thm::psd2} below. 

\begin{theorem}\label{thm::psd2}
For a symmetric $m\times m$ matrix $X$, the following conditions are equivalent: 

(i) $X$ is a PSD matrix; 

(ii) for any consecutive principal submatrix of $X$, one of its inner-saturated submatrices has a nonnegative determinant;

(iii) for any consecutive principal submatrix of $X$, any of its inner-saturated submatrices has a nonnegative determinant. 
\end{theorem}

To appreciate the significance of conditions (ii) and (iii) in Theorem \ref{thm::psd2}, we remark that for a general symmetric matrix $X$, even if one of its inner-saturated submatrices has a nonnegative determinant, it is still possible that its other inner-saturated submatrices have a negative determinant. However, in our Theorem \ref{thm::psd2}, the condition (ii) is equivalent to conditions (i) and (iii). Based on condition (ii), to verify whether or not a symmetric $m\times m$ matrix $X$ is PSD, we only need to check the determinants of $m(m+1)/2$ principal submatrices, including one inner-saturated submatrix of each consecutive principal submatrix of $X$. By Definitions \ref{def::consecutive-principal} and \ref{def::inner-saturated}, any inner-saturated submatrix of a consecutive principal submatrix of $X$ is a principal submatrix of $X$. Therefore, Theorem \ref{thm::psd2} is stronger than the classic Sylvester’s criterion for PSD matrices in Theorem \ref{thm::syl}(ii).

\section{Toward elementwise characterization of PD and PSD matrices}\label{sec2}

\subsection{Theoretical results}

Theorem \ref{thm::psd2} gives a sufficient and necessary condition for a PSD matrix based on all its consecutive principal submatrices. These conditions of consecutive principal submatrices lead to elementwise characterization of PSD matrices. We also have analogous conditions of consecutive principal submatrices which lead to elementwise characterization of PD matrices. In this section, we provide a method to determine the range of elements in PD and PSD matrices, which forms the foundation for the applications in Sections \ref{sec3} and \ref{sec4}. 

We first introduce the following two propositions. They are crucial for determining the range of elements in PD and PSD matrices. 

\begin{proposition}\label{prop::pd}
Consider a symmetric $m\times m$ matrix $X$, where both $X_{1:(m-1),1:(m-1)}$ and $X_{2:m,2:m}$ are PD. (i) $X$ is PD if and only if $X$ has a positive determinant. (ii) For arbitrary values in $X$ except $X_{1,m}$ and $X_{m,1}$, there always exists a real number $k$ such that we can set $X_{1,m}=X_{m,1}=k$ to ensure $X$ is PD. 
\end{proposition}

\begin{proposition}\label{prop::psd}
Consider a symmetric $m\times m$ matrix $X$, where both $X_{1:(m-1),1:(m-1)}$ and $X_{2:m,2:m}$ are PSD. (i) $X$ is PSD if and only if one of the inner-saturated submatrices of $X$ has a nonnegative determinant. (ii) For arbitrary values in $X$ except $X_{1,m}$ and $X_{m,1}$, there always exists a real number $k$ such that we can set $X_{1,m}=X_{m,1}=k$ to ensure $X$ is PSD. 
\end{proposition}

By Proposition \ref{prop::pd} or \ref{prop::psd}, based on the determinant of $X$ or one of its inner-saturated submatrices, we can determine the range of $X_{1,m}=X_{m,1}$ such that $X$ is PD or PSD, if $X_{1:(m-1),1:(m-1)}$ and $X_{2:m,2:m}$ are PD or PSD. Recursively, we can determine the range of $X_{1,m-1}$ and $X_{2,m}$ such that $X_{1:(m-1),1:(m-1)}$ and $X_{2:m,2:m}$ are PD or PSD, respectively, if all of $X_{1:(m-2),1:(m-2)}$, $X_{2:(m-1),2:(m-1)}$ and $X_{3:m,3:m}$ are PD or PSD. The process can proceed recursively to determine the range of all elements in $X$. We give more details in the next two subsections. 

\subsection{General method}

For simplicity of the discussion, we define the $k$-diagonal elements in $X$ as the set $\{X_{i,i+k}: 1\leq i\leq m-k\}$, for $0\leq k\leq m-1$. When $k=0$, the $k$-diagonal elements are exactly the diagonal elements in $X$. 

For a PD matrix, first, we can determine the range of diagonal elements in $X$: they must all be positive. Then, we can determine the range of 1-diagonal elements in $X$, given the diagonal elements in $X$: the element $X_{i,i+1}$ must satisfy that $X_{i:(i+1),i:(i+1)}$ has a positive determinant. We can proceed to determine the range of $k$-diagonal elements, given the previous elements, for $2\leq k\leq m-1$: the element $X_{i,i+k}$ must satisfy that $X_{i:(i+k),i:(i+k)}$ has a positive determinant. 

For a PSD matrix, the procedure is similar. For a general $k\leq m-1$, we can determine the range of $k$-diagonal elements, given the previous elements: the element $X_{i,i+k}$ must satisfy that one inner-saturated submatrix of $X_{i:(i+k),i:(i+k)}$ has a nonnegative determinant.

The above discussion motivates a partial ordering on the upper triangular elements in $X$, denoted as $(\{X_{i,j},i\leq j\}, \preceq)$. We define $X_{i',j'}\preceq X_{i,j}$ if and only if $i\leq i'\leq j'\leq j$. From the previous discussion, when $i\leq j$, we can determine the range of $X_{i,j}$ based on all other elements in $X_{i:j,i:j}$; and any upper triangular element $X_{i',j'}$ in $X_{i:j,i:j}$ must satisfy $i\leq i'\leq j'\leq j$, ensuring that $X_{i,j}\preceq X_{i',j'}$ based on our definition. Equivalently, we can consider a directed acyclic graph. The vertex set contains all the upper triangular elements in $X$. For element $X_{i,j}$ where $i<j$, there is a directed edge from $X_{i,j-1}$ to $X_{i,j}$ and a directed edge from $X_{i+1,j}$ to $X_{i,j}$. See Figure \ref{fig1} below for the graph construction when $m=4$. From the graph, $X_{i',j'}\preceq X_{i,j}$ if and only if $X_{i',j'}$ is a predecessor of $X_{i,j}$. All diagonal elements in $X$ do not have predecessors, and we can determine their range at the beginning. For a general element $X_{i,j}$ with $i<j$, we can determine its range given its predecessors. 

\begin{figure}[h]
\centering
\begin{tikzpicture}[node distance=1.5cm and 1.5cm, 
                    every node/.style={circle, draw, minimum size=1cm, inner sep=2pt}]
    \node (X11) at (0,0) {$X_{1,1}$};
    \node (X22) at (2,0) {$X_{2,2}$};
    \node (X33) at (4,0) {$X_{3,3}$};
    \node (X44) at (6,0) {$X_{4,4}$};
    
    \node (X12) at ($(X11)!0.5!(X22)+(0,1.5)$) {$X_{1,2}$};
    \node (X23) at ($(X22)!0.5!(X33)+(0,1.5)$) {$X_{2,3}$};
    \node (X34) at ($(X33)!0.5!(X44)+(0,1.5)$) {$X_{3,4}$};

    \node (X13) at ($(X12)!0.5!(X23)+(0,1.5)$) {$X_{1,3}$};
    \node (X24) at ($(X23)!0.5!(X34)+(0,1.5)$) {$X_{2,4}$};

    \node (X14) at ($(X13)!0.5!(X24)+(0,1.5)$) {$X_{1,4}$};

    \draw[->] (X11) -- (X12);
    \draw[->] (X22) -- (X12);
    \draw[->] (X22) -- (X23);
    \draw[->] (X33) -- (X23);
    \draw[->] (X33) -- (X34);
    \draw[->] (X44) -- (X34);

    \draw[->] (X12) -- (X13);
    \draw[->] (X23) -- (X13);
    \draw[->] (X23) -- (X24);
    \draw[->] (X34) -- (X24);

    \draw[->] (X13) -- (X14);
    \draw[->] (X24) -- (X14);
\end{tikzpicture}
\caption{Partial ordering of upper triangular elements in a symmetric $4\times 4$ matrix $X$. }
\label{fig1}
\end{figure}
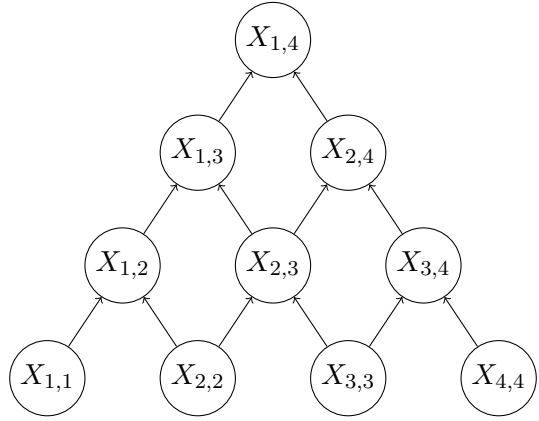

Although the statement of the classic Sylvester's criterion for PD matrices is more concise, our elementwise characterization for PD matrices can be more convenient in some problems, due to the clear mapping between the $m(m+1)/2$ determinant constraints and the $m(m+1)/2$ elements in $X$. Our characterization for PSD matrices is definitely more convenient than classic Sylvester's criterion for PSD matrices especially when $m$ is large. We will illustrate the applications of these results in Sections \ref{sec3} and \ref{sec4}. Before presenting the applications, we discuss a simple problem of determining the range of the element in a $4\times 4$ matrix.

\subsection{Example: $4\times 4$ matrix}
For a PD matrix $X$, the range of elements in $X$ can be determined below: 
\begin{itemize}
\item $X_{1,1},X_{2,2},X_{3,3},X_{4,4}>0$;
\item given $\{X_{1,1},X_{2,2},X_{3,3},X_{4,4}\}$, then $X_{1,2},X_{2,3}, X_{3,4}$ must satisfy $\mathrm{det}(X_{1:2,1:2})>0$, $\mathrm{det}(X_{2:3,2:3})>0$, $\mathrm{det}(X_{3:4,3:4})>0$, respectively; 
\item given $\{X_{1,1},X_{2,2},X_{3,3},X_{4,4},X_{1,2},X_{2,3}, X_{3,4}\}$, then $X_{1,3},X_{2,4}$ must satisfy $\mathrm{det}(X_{1:3,1:3})>0$, $\mathrm{det}(X_{2:4,2:4})>0$, respectively; 
\item given $\{X_{1,1},X_{2,2},X_{3,3},X_{4,4},X_{1,2},X_{2,3}, X_{3,4},X_{1,3},X_{2,4}\}$, then $X_{1,4}$ must satisfy $\mathrm{det}(X)>0$. 
\end{itemize}
For a PSD matrix $X$, assuming $X_{2,2}$ and $X_{3,3}$ are positive, the range of elements in $X$ can be determined below: 
\begin{itemize}
\item $X_{1,1},X_{2,2},X_{3,3},X_{4,4}\geq 0$;
\item given $\{X_{1,1},X_{2,2},X_{3,3},X_{4,4}\}$, then $X_{1,2},X_{2,3}, X_{3,4}$ must satisfy $\mathrm{det}(X_{1:2,1:2})\geq 0$, $\mathrm{det}(X_{2:3,2:3})\geq 0$, $\mathrm{det}(X_{3:4,3:4})\geq 0$, respectively; 
\item given $\{X_{1,1},X_{2,2},X_{3,3},X_{4,4},X_{1,2},X_{2,3}, X_{3,4}\}$, then $X_{1,3},X_{2,4}$ must satisfy $\mathrm{det}(X_{1:3,1:3})\geq 0$, $\mathrm{det}(X_{2:4,2:4})\geq 0$, respectively; 
\item given $\{X_{1,1},X_{2,2},X_{3,3},X_{4,4},X_{1,2},X_{2,3}, X_{3,4},X_{1,3},X_{2,4}\}$, then $X_{1,4}$ must satisfy $\mathrm{det}(X)\geq 0$ when $\mathrm{det}(X_{2:3,2:3})\not=0$, and $\mathrm{det}(X_{\{1,2,4\},\{1,2,4\}})\geq 0$ when $\mathrm{det}(X_{2:3,2:3})=0$. 
\end{itemize}
For the edge case with $X_{2,2}=0$, a PSD matrix $X$ must satisfy that $X_{\{1,3,4\},\{1,3,4\}}$ is PSD and $X_{2,i}=X_{2,i}=0$ for $1\leq i\leq 4$. To ensure that $X_{\{1,3,4\},\{1,3,4\}}$ is PSD, the range of elements in $X_{\{1,3,4\},\{1,3,4\}}$ can be determined below:
\begin{itemize}
\item $X_{1,1},X_{3,3},X_{4,4}\geq 0$;
\item given $\{X_{1,1},X_{3,3},X_{4,4}\}$, then $X_{1,3}, X_{3,4}$ must satisfy $\mathrm{det}(X_{\{1,3\},\{1,3\}})\geq 0$, $\mathrm{det}(X_{3:4,3:4})\geq 0$, respectively; 
\item given $\{X_{1,1},X_{3,3},X_{4,4},X_{1,3},X_{3,4}\}$, then $X_{1,4}$ must satisfy $\mathrm{det}(X)\geq 0$ when $X_3\not=0$, and $\mathrm{det}(X_{\{1,4\},\{1,4\}})\geq 0$ when $\mathrm{det}(X_{2:3,2:3})=0$. 
\end{itemize}
The edge case with $X_{3,3}=0$ is similar.

\section{Application 1: PD and PSD matrix completion}\label{sec3}

In this section, we focus on the problem of completing a partially observed symmetric matrix $X$ to ensure it is PD or PSD. Given that not all entries of $X$ are observed, a fundamental problem is to determine whether the missing entries can be completed in such a way that $X$ becomes PD or PSD. This problem, known as PD or PSD matrix completion, been studied previously \citep{grone1984positive, barrett1989determinantal, johnson1990matrix}. Specifically, they associate the observed pattern of an $m\times m$ matrix $X$ with a graph $G$, where $G$ has $m$ vertices $1:m$, with $i$ and $j$ connected if $X_{i,j}$ is observed. In graph theory, a graph is chordal if every cycle with four or more vertices has a chord. Here a cycle is a closed path of distinct vertices, and a chord is an edge connecting two non-adjacent vertices in the cycle. When $G$ is a chordal graph, \citet{grone1984positive} prove that $X$ can be completed to a PD or PSD matrix as long as all the fully observed principal submatrices of $X$ are PD or PSD. However, when $G$ is not a chordal graph, it remains unclear how to quickly determine if a general partially observed symmetric matrix $X$ has a PD or PSD completion.

For simplicity of the presentation, we first discuss PD completion. Based on Section \ref{sec2}, our approach is to determine if the range of missing entries in $X$ is empty when $X$ is PD. We begin with a specific example below and then give a more general discussion. 

\begin{example}
Consider
\begin{equation*}
X=\begin{pmatrix}
1 & 0.8 & 0.6 & 0.8 & x_3 \\
0.8 & 1 & 0.4 & x_2 & 0.5 \\
0.6 & 0.4 & 1 & x_1 & 0.6 \\
0.8 & x_2 & x_1 & 1 & 0.9 \\
x_3 & 0.5 & 0.6 & 0.9 & 1 \\
\end{pmatrix},
\end{equation*}
where $x_1,x_2,x_3$ are missing entries in $X$. The associated graph of $X$ is not a chordal graph. This is because $(1,4,5,2)$ forms a cycle with four vertices but lacks a chord in the associated graph.  

By Proposition \ref{prop::pd}, we only need to determine the range of $(x_1,x_2)$ such that both $X_{1:4,1:4}$ and $X_{2:5,2:5}$ are PD, because there always exists a real number $x_3$ such that $X$ is PD, provided that both $X_{1:4,1:4}$ and $X_{2:5,2:5}$ are PD. By exchanging rows and columns, we take $Y=X_{\{2,3,1,4\},\{2,3,1,4\}}$ and $Z=X_{\{2,3,5,4\},\{2,3,5,4\}}$ as follows: 

\begin{equation*}
Y=\begin{pmatrix}
1 & 0.4 & 0.8 & x_2 \\
0.4 & 1 & 0.6 & x_1 \\
0.8 & 0.6 & 1 & 0.8 \\
x_2 & x_1 & 0.8 & 1 \\
\end{pmatrix},
\quad 
Z=\begin{pmatrix}
1 & 0.4 & 0.5 & x_2 \\
0.4 & 1 & 0.6 & x_1 \\
0.5 & 0.6 & 1 & 0.9 \\
x_2 & x_1 & 0.9 & 1 \\
\end{pmatrix}.
\end{equation*}
Since $Y_{1:3,1:3}$ and $Y_{3:4,3:4}$ are PD, to ensure $Y$ is PD, $(x_1,x_2)$ must satisfy $\mathrm{det}(Y_{2:4,2:4})>0$ and $\mathrm{det}(Y)>0$. Moreover, we can express $\mathrm{det}(Y_{2:4,2:4})$ as a function of $x_1$ and $\mathrm{det}(Y)$ as a function of $(x_1,x_2)$ below
\begin{equation*}
\begin{split}
\mathrm{det}(Y_{2:4,2:4})&=-x_1^2+0.96x_1,\\
\mathrm{det}(Y)&= -0.36x_1^2 -0.64x_2^2 -0.16x_1x_2 + 0.448x_1+0.896x_2-0.3136.
\end{split}
\end{equation*}
Thus, $\mathrm{det}(Y_{2:4,2:4})>0$ implies that the range of $x_1$ is $(0,0.96)$, and $\mathrm{det}(Y)>0$ implies that given $x_1$, the range of $x_2$ is 
$$\left(\frac{(0.16x_1-0.896)-\sqrt{0.224(-x_1^2+0.96x_1)}}{1.28},\frac{(0.16x_1-0.896)+\sqrt{0.224(-x_1^2+0.96x_1)}}{1.28}\right).$$
Similarly, since $Z_{1:3,1:3}$ and $Z_{3:4,3:4}$ are PD, to ensure $Z$ is PD, $(x_1,x_2)$ must satisfy $\mathrm{det}(Z_{2:4,2:4})>0$ and $\mathrm{det}(Z)>0$. Moreover, we can express $\mathrm{det}(Z_{2:4,2:4})$ as a function of $x_1$ and $\mathrm{det}(Z)$ as a function of $(x_1,x_2)$ below
\begin{equation*}
\begin{split}
\mathrm{det}(Z_{2:4,2:4})&=-x_1^2+1.08x_1-0.17,\\
\mathrm{det}(Z)&= - 0.75x_1^2 -0.64x_2^2 + 0.2x_1x_2 + 0.72x_1+0.468x_2-0.2104.
\end{split}
\end{equation*}
Thus, $\mathrm{det}(Z_{2:4,2:4})>0$ implies that the range of $x_1$ is $(0.54-0.08\sqrt{19},0.54+0.08\sqrt{19})$, and $\mathrm{det}(Z)>0$ implies that given $x_1$, the range of $x_2$ is 
$$\left(\frac{(-0.2x_1-0.468)-\sqrt{0.47(-x_1^2+1.08x_1-0.17)}}{1.28},\frac{(-0.2x_1-0.468)+\sqrt{0.47(-x_1^2+1.08x_1-0.17)}}{1.28}\right).$$

We plot the two ranges of $(x_1,x_2)$ in Figure \ref{fig2} below. Since there is an nonempty intersection of the two ranges, we conclude that $X$ has a PD completion. 

\begin{figure}[h]
\centering
\includegraphics[width=0.8\textwidth]{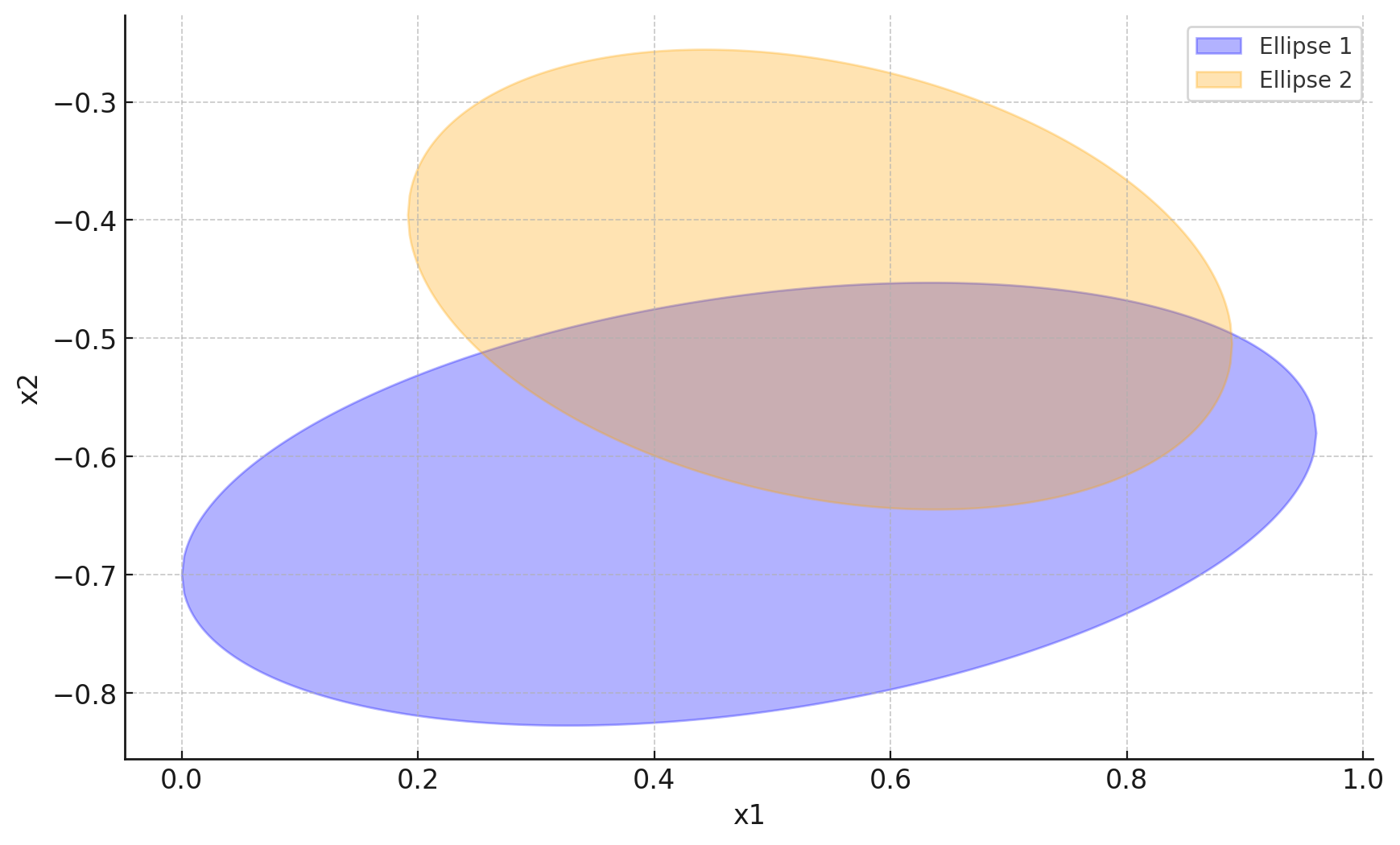}
\caption{Feasible region of $(x_1,x_2)$ as the intersection of the inner parts of two ellipses to ensure $X$ has a PD completion.}\label{fig2}
\end{figure}
\end{example}

\subsection{A general procedure for PD matrix completion}\label{sec3::pd}

Consider a general partially observed symmetric $m\times m$ matrix $X$, consisting of missing values $x_1,x_2,...,x_K$. If $X_{i,i}$ is missing, we can delete the $i$th row and the $i$th column of $X$ to examine whether the remaining submatrix has a PD completion because we can choose $X_{i,i}$ to be large enough such that the entire matrix is PD. Therefore, we assume all the diagonal elements in $X$ are observed. To obtain the range of $(x_1,...,x_K)$ such that $X$ is PD, we consider the following steps: 
\begin{enumerate}
\item interchange rows and columns in $X$ such that $X_{1,m}=X_{m,1}=x_K$;
\item determine the range of $(x_1,...,x_{K-1})$ such that both $X_{1:(m-1),1:(m-1)}$ and $X_{2:m,2:m}$ are PD;
\item determine the range of $x_K$ such that $\mathrm{det}(X)>0$, given each $(x_1,...,x_{K-1})$.
\end{enumerate}

Since the second step involves determining the range of $(x_1,...,x_{K-1})$ such that both $X_{1:(m-1),1:(m-1)}$ and $X_{2:m,2:m}$ are PD, we need to repeat the above steps recursively.  Generally, we need to search over the range of $(x_1,...,x_k)$ such that some submatrices of $X$ are PD, for $1\leq k\leq K-1$. The range of $(x_1,...,x_k)$ at each step must be convex, due to the convexity of the space of PD matrices and the convexity of intersection of finite convex sets. When $K$ is small, we recommend the grid search to determine the range of $(x_1,...,x_k)$ for each $1\leq k\leq K-1$. When $K$ is large, we can use some tree-based methods to improve the computational efficiency. The computational issues are interesting research topics but they are beyond the scope of our current paper.

\subsection{A general procedure for PSD matrix completion}
The discussion of PSD completion is similar but slightly more complicated. We need to modify the procedure in Section \ref{sec3::pd}: 

\begin{enumerate}
\item interchange rows and columns in $X$ such that $X_{1,m}=X_{m,1}=x_K$;
\item determine the range of $(x_1,...,x_{K-1})$ such that both $X_{1:(m-1),1:(m-1)}$ and $X_{2:m,2:m}$ are PSD;
\item determine the range of $x_K$ such that one inner-saturated submatrix of $X$ has a nonnegative determinant, given each $(x_1,...,x_{K-1})$.
\end{enumerate}

Similarly, since the second step involves determining the range of $(x_1,...,x_{K-1})$ such that both $X_{1:(m-1),1:(m-1)}$ and $X_{2:m,2:m}$ are PSD, we need to repeat the above steps recursively. The range of $(x_1,...,x_k)$ at each step must also be convex, and when $K$ is small, we recommend the grid search to determine the range of $(x_1,...,x_k)$ for each $1\leq k\leq K-1$. Again, we omit the computational details. 

%

\section{Application 2: nonlinear semidefinite program}\label{sec4}

Another direct application of our result is nonlinear semidefinite program. We consider the optimization problem of the form:
\begin{equation*}
\begin{aligned}
&\min \quad &&f(X)\\
&\text{subject to }\quad &&X \text{ is PSD}, \quad g(X)\geq 0, \quad h(X)=0,
\end{aligned}
\end{equation*}
where $g$ and $h$ can be vector functions of the matrix $X$. When $f,g,h$ are linear functions of $X$, such as $f(X)=\mathrm{trace}(AX)$ for some matrices $A$, the problem is a standard semidefinite program. When $f,g,h$ are general functions of elements in $X$, this problem becomes a nonlinear semidefinite program, which is more challenging to solve. Typical solutions include augmented Lagrangian method, sequential linear or quadratic programming method, and interior point methods; see a recent survey paper \citet{yamashita2015survey} for detailed reviews. Our result provides a reparameterization to nonlinear semidefinite program by translating the PSD constraint into elementwise constraints. 

We discuss the case when $m=4$. To address the constraint that $X$ is PSD, we consider the following 5 cases and solve the optimization problems under these constraints, respectively. 
\begin{itemize}
\item $X_{1,1},X_{4,4}\geq 0$, $X_{2,2},X_{3,3}>0$, $\mathrm{det}(X_{2:3,2:3})>0$, $\mathrm{det}(X_{1:3,1:3})\geq 0$, $\mathrm{det}(X_{2:4,2:4})\geq 0$, $\mathrm{det}(X)\geq 0$. 
\item $X_{1,1},X_{4,4}\geq 0$, $X_{2,2},X_{3,3}>0$, $\mathrm{det}(X_{2:3,2:3})=0$, $\mathrm{det}(X_{1:3,1:3})\geq 0$, $\mathrm{det}(X_{2:4,2:4})\geq 0$, $\mathrm{det}(X_{\{1,2,4\},\{1,2,4\}})\geq 0$. 
\item $X_{1,1},X_{4,4}\geq 0$, $X_{2,2}>0$, $\mathrm{det}(X_{1:2,1:2})\geq 0$, $\mathrm{det}(X_{\{2,4\},\{2,4\}})\geq 0$, $\mathrm{det}(X_{\{1,2,4\},\{1,2,4\}})\geq 0$, $X_{3,3}=X_{1,3}=X_{2,3}=X_{3,4}=0$, 
\item $X_{1,1},X_{4,4}\geq 0$, $X_{3,3}>0$, $\mathrm{det}(X_{\{1,3\},\{1,3\}})\geq 0$, $\mathrm{det}(X_{3:4,3:4})\geq 0$, $\mathrm{det}(X_{\{1,3,4\},\{1,3,4\}})\geq 0$, $X_{2,2}=X_{1,2}=X_{2,3}=X_{2,4}=0$, 
\item $X_{1,1},X_{4,4}\geq 0$, $\mathrm{det}(X_{\{1,4\},\{1,4\}})\geq 0$, $X_{2,2}=X_{3,3}=X_{1,2}=X_{1,3}=X_{2,3}=X_{2,4}=X_{3,4}=0$. 
\end{itemize} 
The above five cases cover all the possibilities that $X$ is PSD, and we only need to find the minimum solution of the corresponding five optimization problems. We provide a following concrete example below. 

\begin{example}
We consider the optimization problem
\begin{equation*}
\begin{aligned}
&\max \quad &&X_{11}^2+X_{22}^2+X_{33}^2+X_{44}^2-X_{12}X_{23}X_{34}-X_{13}X_{24}+X_{14}\\
&\text{subject to }\quad &&X \text{ is PSD}, \quad 0\leq X_{11},X_{22},X_{33},X_{44}\leq 1.
\end{aligned}
\end{equation*}
Notice that for any PSD matrix $X$ satisfying $0\leq X_{11},X_{22},X_{33},X_{44}\leq 1$, we can update the diagonal elements of $X$ to be all equal to 1, where the new matrix also satisfies the PSD constraint and increases the value in the objective function. Therefore, it is equivalent to solving the following optimization problem
\begin{equation*}
\begin{aligned}
&\max \quad &&4-X_{12}X_{23}X_{34}-X_{13}X_{24}+X_{14}\\
&\text{subject to }\quad &&X \text{ is PSD}, \quad X_{11}=X_{22}=X_{33}=X_{44}=1.
\end{aligned}
\end{equation*}
Next, we translate the PSD constraint into the following elementwise constraints: 
\begin{itemize}
\item $-1\leq X_{12},X_{23},X_{34}\leq 1$;
\item $\mathrm{det}(X_{1:3,1:3})\geq 0$, $\mathrm{det}(X_{2:4,2:4})\geq 0$;
\item when $|X_{23}|\not=1$, we have $\mathrm{det}(X)\geq 0$; when $|X_{23}|=1$, we have $\mathrm{det}(X_{\{1,2,4\},\{1,2,4\}})\geq 0$. 
\end{itemize}
To address the two cases corresponding to the third item above, we compute the following two optimization problems: 
\begin{equation*}
\begin{aligned}
&\max \quad &&4-X_{12}X_{23}X_{34}-X_{13}X_{24}+X_{14}\\
&\text{subject to }\quad && X_{11}=X_{22}=X_{33}=X_{44}=1, -1\leq X_{12},X_{34}\leq 1, -1<X_{23}<1\\
&\quad &&\mathrm{det}(X_{1:3,1:3})\geq 0, \mathrm{det}(X_{2:4,2:4})\geq 0, \mathrm{det}(X)\geq 0
\end{aligned}
\end{equation*}
and
\begin{equation*}
\begin{aligned}
&\max \quad &&4-X_{12}X_{23}X_{34}-X_{13}X_{24}+X_{14}\\
&\text{subject to }\quad && X_{11}=X_{22}=X_{33}=X_{44}=1,-1\leq X_{12},X_{34}\leq 1, |X_{23}|=1\\
&\quad &&\mathrm{det}(X_{1:3,1:3})\geq 0, \mathrm{det}(X_{2:4,2:4})\geq 0, \mathrm{det}(X_{\{1,2,4\},\{1,2,4\}})\geq 0.
\end{aligned}
\end{equation*}
They are optimization problems involving six variables with nonlinear objective function and nonlinear constraints. We omit the detail of this particular optimization problem.
\end{example}

We have illustrated our theory with $m=4$. With general $m$, the fundamental idea remains the same but the details become more complicated. The nonlinear semidefinite program plays important roles in many statistical problems, including a recent study on sensitivity analysis in mediation analysis by \citet{zhang2022interpretable}, which motivated this paper.

\section{Proofs}

\subsection{Lemmas}
\begin{lemma}\label{lem::det-quad}
For a symmetric $m\times m$ matrix $X$ with $m\geq 2$, $\mathrm{det}(X)$ is a quadratic function of $X_{1,m}$, given all other elements in $X$. In particular, $\mathrm{det}(X)=aX_{1,m}^2+bX_{1,m}+c$, where
\begin{equation*}
\begin{split}
a&=\begin{cases}
-1 & \text{if }m=2,\\
-\mathrm{det}(X_{2:(m-1),2:(m-1)}) & \text{if }m\geq 3,\\
\end{cases}\\
b&=2\times (-1)^{m+1}\times \mathrm{det}\begin{pmatrix}
X_{1,2:(m-1)} & 0\\
X_{2:(m-1),2:(m-1)} & X_{2:(m-1),m} \\
\end{pmatrix},\\
c&=\mathrm{det}\begin{pmatrix}
X_{1,1} & X_{1,2:(m-1)} & 0\\
X_{2:(m-1),1} & X_{2:(m-1),2:(m-1)} & X_{2:(m-1),m} \\
0 & X_{m,2:(m-1)} & X_{m,m}\\
\end{pmatrix}.\\
\end{split}
\end{equation*}
Furthermore, $b^2-4ac=4\mathrm{det}(X_{1:(m-1),1:(m-1)})\mathrm{det}(X_{2:m,2:m})$. 
\end{lemma}

\begin{proof}[Proof of Lemma \ref{lem::det-quad}]
We can derive the expressions of $a,b,c$ from the definition of determinant, but omit the details. We only verify the formula of $b^2 - 4ac$. Define $A=X_{2:(m-1),2:(m-1)}$, $x=X_{2:(m-1),1}$, $y=X_{2:(m-1),m}$, $z_1=X_{1,1}$, $z_2=X_{m,m}$, and partition $X$ as
\begin{equation*}
X=\begin{pmatrix}
z_1 & x^\T & X_{1,m} \\
x & A & y \\
X_{1,m} & y^\T & z_2
\end{pmatrix}.
\end{equation*}

\paragraph{We first consider the case when $A$ is invertible.} We have $a=-\mathrm{det}(A)\not=0$. Since
\begin{equation}
\label{eq::gauss1}
G_1G_2
XG_2^\T G_1^{\T}=\begin{pmatrix}
z_1-x^\T A^{-1}x & 0 & X_{1,m}-x^\T A^{-1}y \\
0 & A & 0 \\
X_{1,m}-x^\T A^{-1}y & 0 & z_2-y^\T A^{-1}y
\end{pmatrix},
\end{equation}
where 
\begin{equation*}
G_1=\begin{pmatrix}
1 & 0 & 0 \\
0 & I & 0 \\
0 & -y^\T A^{-1} & 1
\end{pmatrix},\quad 
G_2=\begin{pmatrix}
1 & -x^\T A^{-1} & 0 \\
0 & I & 0 \\
0 & 0 & 1
\end{pmatrix}
\end{equation*}
have determinants equal to 1, we have
\begin{equation}\label{eq::formula1}
\begin{split}
c=&\mathrm{det}\begin{pmatrix}
z_1-x^\T A^{-1}x & 0 & -x^\T A^{-1}y \\
0 & A & 0 \\
-x^\T A^{-1}y & 0 & z_2-y^\T A^{-1}y
\end{pmatrix}\\
=&\left\{(z_1-x^\T A^{-1}x)(z_2-y^\T A^{-1}y)-(x^\T A^{-1}y)^2\right\}\mathrm{det}(A).
\end{split}
\end{equation}
Similarly, since
\begin{align}
&\begin{pmatrix}
1 & -x^\T A^{-1} \\
0 & I \\
\end{pmatrix}
\begin{pmatrix}
z_1 & x^\T \\
x & A \\
\end{pmatrix}\begin{pmatrix}
1 & 0 \\
-A^{-1}x  & I \\
\end{pmatrix} =\begin{pmatrix}
z_1-x^\T A^{-1}x & 0 \\
0 & A \\
\end{pmatrix},\label{eq::gauss2}\\
&\begin{pmatrix}
I & 0 \\
-y^\T A^{-1} & 1
\end{pmatrix}
\begin{pmatrix}
A & y \\
y^\T & z_2
\end{pmatrix}\begin{pmatrix}
I & -A^{-1}y \\
0 & 1
\end{pmatrix} =\begin{pmatrix}
A & 0 \\
0 & z_2-y^\T A^{-1}y \\
\end{pmatrix},\label{eq::gauss3}\\
&\begin{pmatrix}
1 & -x^\T A^{-1} \\
0 & I \\
\end{pmatrix}
\begin{pmatrix}
x^\T & 0 \\
A & y
\end{pmatrix} =\begin{pmatrix}
0 & -x^\T A^{-1}y\\
A & y \\
\end{pmatrix},\label{eq::gauss4}
\end{align}
we have
\begin{align}
&\mathrm{det}(X_{1:(m-1),1:(m-1)})=(z_1-x^\T A^{-1}x)\mathrm{det}(A),\label{eq::formula2}\\
&\mathrm{det}(X_{2:m,2:m})=(z_2-y^\T A^{-1}y)\mathrm{det}(A),\label{eq::formula3}\\
&b=2\times (-1)^{m+1}\times \mathrm{det}\begin{pmatrix}
0 & -x^\T A^{-1}y\\
A & y \\
\end{pmatrix}=-2(x^\T A^{-1}y)\mathrm{det}(A). \label{eq::formula4}
\end{align}
Based on \eqref{eq::formula1} and \eqref{eq::formula2}--\eqref{eq::formula4}, 
\begin{equation*}
\begin{split}
b^2-4ac&=\left[4(x^\T A^{-1}y)^2+4\left\{(z_1-x^\T A^{-1}x)(z_2-y^\T A^{-1}y)-(x^\T A^{-1}y)^2\right\}\right]\mathrm{det}(A)^2\\
&=4\mathrm{det}(X_{1:(m-1),1:(m-1)})\mathrm{det}(X_{2:m,2:m}).
\end{split}
\end{equation*}

\paragraph{We then consider the case when $A$ is not invertible.} We consider $\tilde{X}=X+\epsilon I$ such that $\tilde{X}$ is invertible. We define $a(\epsilon),b(\epsilon),c(\epsilon)$ accordingly. Then, $a(\epsilon),b(\epsilon),c(\epsilon)$, $\mathrm{det}(\tilde{X}_{1:(m-1),1:(m-1)})$, $\mathrm{det}(\tilde{X}_{2:m,2:m})$ are continuous functions of $\epsilon$. Since we have 
$$b(\epsilon)^2-4a(\epsilon)c(\epsilon)=4\mathrm{det}(\tilde{X}_{1:(m-1),1:(m-1)})\mathrm{det}(\tilde{X}_{2:m,2:m})$$
for arbitrarily $\epsilon$ near 0, as $\epsilon\to 0$, we have 
$$b^2-4ac=4\mathrm{det}(X_{1:(m-1),1:(m-1)})\mathrm{det}(X_{2:m,2:m}).$$
\end{proof}

\begin{lemma}\label{lem::psd-suff1}
For $m\geq 3$, a symmetric $m\times m$ matrix $X$ is PSD if the following three conditions hold: 
\begin{enumerate}
\item[(i)] $X_{1:(m-1),1:(m-1)}$ and $X_{2:m,2:m}$ are PSD; 
\item[(ii)] $X_{2:(m-1),2:(m-1)}$ is PD; 
\item[(iii)] $\mathrm{det}(X)\geq 0$. 
\end{enumerate}
\end{lemma}

\begin{proof}[Proof of Lemma \ref{lem::psd-suff1}]
We follow the notation in the proof of Lemma \ref{lem::det-quad}, with $A=X_{2:(m-1),2:(m-1)}$, $x=X_{2:(m-1),1}$, $y=X_{2:(m-1),m}$, $z_1=X_{1,1}$, $z_2=X_{m,m}$. Since $X_{1:(m-1),1:(m-1)}$ and $X_{2:m,2:m}$ are PSD, by \eqref{eq::gauss2}--\eqref{eq::gauss3}, we have $z_1-x^\T A^{-1}x\geq 0$, $z_2-y^\T A^{-1}y\geq 0$. Moreover, $\mathrm{det}(X)\geq 0$ and $\mathrm{det}(A)>0$ imply that 
\begin{equation*}
\begin{pmatrix}
z_1-x^\T A^{-1}x  & X_{1,m}-x^\T A^{-1}y \\
X_{1,m}-x^\T A^{-1}y  & z_2-y^\T A^{-1}y
\end{pmatrix}
\end{equation*}
is PSD and thus
\begin{equation*}
\begin{pmatrix}
z_1-x^\T A^{-1}x & 0 & X_{1,m}-x^\T A^{-1}y \\
0 & A & 0 \\
X_{1,m}-x^\T A^{-1}y & 0 & z_2-y^\T A^{-1}y
\end{pmatrix},
\end{equation*}
is PSD. By \eqref{eq::gauss1}, we conclude that $X$ is PSD. 

\end{proof}

\begin{lemma}\label{lem::inner-saturated}
For a symmetric $m\times m$ matrix $X$, assume $X_{1:(m-1),1:(m-1)}$ and $X_{2:m,2:m}$ are PSD and $X_{2:(m-1),2:(m-1)}$ is not invertible. Then, we have the following results: 

(i) $\mathrm{det}(X)=0$. 

(ii) Suppose $Y_1,Y_2$ are possibly different inner-saturated submatrices of $X$. Write $\mathrm{det}(Y_1)$ and $\mathrm{det}(Y_2)$ as quadratic functions of $X_{1,m}$, given all other elements in $Y_1$ and $Y_2$. In particular, $\mathrm{det}(Y_1)=a_1X_{1,m}^2+b_1X_{1,m}+c_1$ and $\mathrm{det}(Y_2)=a_2X_{1,m}^2+b_2X_{1,m}+c_2$. Then, there exists a positive constant $k$ such that $(a_1,b_1,c_1)=k(a_2,b_2,c_2)$. 
\end{lemma}

\begin{proof}[Proof of Lemma \ref{lem::inner-saturated}]
We follow the notation in the proof of Lemma \ref{lem::det-quad}, with $A=X_{2:(m-1),2:(m-1)}$, $x=X_{2:(m-1),1}$, $y=X_{2:(m-1),m}$, $z_1=X_{1,1}$, $z_2=X_{m,m}$. Write 
\begin{equation*}Y_1=
\begin{pmatrix}
z_1 & x_1^\T & X_{1,m} \\
x_1 & A_1 & y_1 \\
X_{1,m} & y_1^\T  & z_2
\end{pmatrix},
\end{equation*}
where $A_1$ is an invertible matrix block in $A$ satisfying $\mathrm{rank}(A_1)=\mathrm{rank}(A)$. To show (ii), we only need to show that $\mathrm{det}(Y_1)/\mathrm{det}(A_1)$ only depends on $A,x,y,z_1,z_2,X_{1,m}$, and does not depend on $A_1,x_1,y_1$. Without loss of generality, we write
\begin{equation*}X=
\begin{pmatrix}
z_1 & x_1^\T & x_2^\T & X_{1,m} \\
x_1 & A_1 & A_2 & y_1 \\
x_2 & A_2^\T & A_3 & y_2 \\
X_{1,m} & y_1^\T & y_2^\T & z_2
\end{pmatrix}. 
\end{equation*}

\paragraph{Proof of Lemma \ref{lem::inner-saturated}(i).}

Using Gaussian elimination, we can multiply on the left and right by matrices with determinants $1$, yielding
\begin{equation*}
\mathrm{det}(X)=\mathrm{det}
\begin{pmatrix}
z_1-x_1^\T A_1^{-1}x_1 & 0 & x_2^\T-x_1^\T A_1^{-1}A_2 & X_{1,m}-X_1^\T A_1^{-1}y_1 \\
0 & A_1 & 0 & 0 \\
x_2-A_2^\T A^{-1}x_1 & 0 & A_3-A_2^\T A_1^{-1}A_2 & y_2-A_2^\T A_1^{-1}y_1 \\
X_{1,m}-X_1^\T A_1^{-1}y_1 & 0 & y_2^\T-y_1^\T A_1^{-1}A_2 & z_2-y_1^\T A_1^{-1}y_1
\end{pmatrix}.
\end{equation*}
Moreover, $A_3-A_2^\T A_1^{-1}A_2=0$ must hold because $\mathrm{rank}(A_1)=\mathrm{rank}(A)$. Also, the condition that both $X_{1:(m-1),1:(m-1)}$ and $X_{2:m,2:m}$ are PSD implies that the following two matrices 
\begin{equation*}
\begin{pmatrix}
z_1-x_1^\T A_1^{-1}x_1 & 0 & x_2^\T-x_1^\T A_1^{-1}A_2  \\
0 & A_1 & 0  \\
x_2-A_2^\T A^{-1}x_1 & 0 & A_3-A_2^\T A_1^{-1}A_2  \\
\end{pmatrix},\quad \begin{pmatrix}
A_1 & 0 & 0 \\
0 & A_3-A_2^\T A_1^{-1}A_2 & y_2-A_2^\T A_1^{-1}y_1 \\
0 & y_2^\T-y_1^\T A_1^{-1}A_2 & z_2-y_1^\T A_1^{-1}y_1
\end{pmatrix}.
\end{equation*}
are PSD. Therefore, we have $x_2-A_2^{\T}A_1^{-1}x_1=0$ and $y_2-A_2^\T A_1^{-1}y_1=0$, which implies $\mathrm{det}(X)=0$ as stated in part (i). 

\paragraph{Proof of Lemma \ref{lem::inner-saturated}(ii).}
Consider $\tilde{X}=X+\epsilon I$ for $\epsilon>0$. Next, we apply Gaussian elimination for $X+\epsilon I$ to obtain the matrix
\begin{equation*}Y(\epsilon)=
\begin{pmatrix}
W_1(\epsilon) & 0 & 0 & W_2(\epsilon) \\
0 & \tilde{A}_1(\epsilon) & 0 & 0 \\
0 & 0 & \tilde{A}_3(\epsilon)  & 0 \\
W_2(\epsilon) & 0 & 0 & W_3(\epsilon)
\end{pmatrix},
\end{equation*}
where 
\begin{equation*}
\begin{split}
\tilde{A}_1(\epsilon)&=A_1+\epsilon I,\\
\tilde{A}_3(\epsilon)&=A_3+\epsilon I-A_2^\T\tilde{A}_1(\epsilon)^{-1}A_2,\\
W_1(\epsilon)&=z_1+\epsilon-x_1^\T\tilde{A}_1^{-1}(\epsilon)x_1-\tilde{x}_2^\T (\epsilon)\tilde{A}_3^{-1}(\epsilon)\tilde{x}_2(\epsilon),\\
W_2(\epsilon)&=X_{1,m}-x_1^\T\tilde{A}_1^{-1}(\epsilon)y_1-\tilde{x}_2^\T (\epsilon)\tilde{A}_3^{-1}(\epsilon)\tilde{y}_2(\epsilon),\\
W_3(\epsilon)&=z_2+\epsilon-y_1^\T\tilde{A}_1^{-1}(\epsilon)y_1-\tilde{y}_2^\T(\epsilon)\tilde{A}_3^{-1}(\epsilon)\tilde{y}_2(\epsilon),\\
\end{split}
\end{equation*}
and $\tilde{x}_2(\epsilon)=x_2-A_2^{\T}\tilde{A}_1(\epsilon)^{-1}x_1$, $\tilde{y}_2(\epsilon)=y_2-A_2^{\T}\tilde{A}_1(\epsilon)^{-1}y_1$. Here we consider an $\epsilon>0$ to ensure that $\tilde{A}_3(\epsilon)$ is invertible. Our proof consists of the following two steps. The first step is to show that $W_1(\epsilon),W_2(\epsilon),W_3(\epsilon)$ only depends on $A,x,y,z_1,z_2,X_{1,m}$. The second step is to show that $\mathrm{det}(Y_1)/\mathrm{det}(A_1)$ only depends on $X_{1,m}$ and the limits of $W_1(\epsilon),W_2(\epsilon),W_3(\epsilon)$, as $\epsilon\to0$. Combining the two steps, we finish showing that $\mathrm{det}(Y_1)/\mathrm{det}(A_1)$ only depends on $A,x,y,z_1,z_2,X_{1,m}$, and thus finish the proof of part (ii). 

\textbf{Step 1}. By Gaussian elimination, we have $\mathrm{det}(X+\epsilon I)=\mathrm{det}(Y(\epsilon))$ and
\begin{equation*}
\mathrm{det}(A+\epsilon I)=\mathrm{det}\begin{pmatrix}
\tilde{A}_1(\epsilon) & 0 \\
0 & \tilde{A}_3(\epsilon)  \\
\end{pmatrix},
\end{equation*}
\begin{equation*}
\mathrm{det}\begin{pmatrix}
z_1+\epsilon & x^\T \\
x & A+\epsilon I 
\end{pmatrix}=\mathrm{det}\begin{pmatrix}
W_1(\epsilon) & 0 & 0  \\
0 & \tilde{A}_1(\epsilon) & 0 \\
0 & 0 & \tilde{A}_3(\epsilon)  \\
\end{pmatrix},
\end{equation*}
\begin{equation*}
\mathrm{det}\begin{pmatrix}
A+\epsilon I & y \\
y^\T & z_2+\epsilon
\end{pmatrix}=\mathrm{det}\begin{pmatrix}
\tilde{A}_1(\epsilon) & 0 & 0 \\
0 & \tilde{A}_3(\epsilon)  & 0 \\
0 & 0 & W_3(\epsilon)
\end{pmatrix}.
\end{equation*}
The above three displayed identities imply
\begin{equation*}
W_1(\epsilon)=\mathrm{det}\begin{pmatrix}
z_1+\epsilon & x^\T \\
x & A+\epsilon I 
\end{pmatrix}\bigg/\mathrm{det}(A+\epsilon I)
\end{equation*}
\begin{equation*}
W_3(\epsilon)=\mathrm{det}\begin{pmatrix}
A+\epsilon I & y \\
y^\T & z_2+\epsilon
\end{pmatrix}\bigg/\mathrm{det}(A+\epsilon I),
\end{equation*}
and 
\begin{equation*}
\mathrm{det}\begin{pmatrix}
W_1(\epsilon) & W_2(\epsilon) \\
W_2(\epsilon) & W_3(\epsilon) \\
\end{pmatrix}=\mathrm{det}(X+\epsilon I)/\mathrm{det}(A+\epsilon I)
\end{equation*}
only depend on $A,x,y,z_1,z_2,X_{1,m}$. Then, $|W_2(\epsilon)|$ only depend on $A,x,y,z_1,z_2,X_{1,m}$. The sign of $W_2(\epsilon)$ can be determined by looking at the change of $\mathrm{det}(X+\epsilon I)/\mathrm{det}(A+\epsilon I)$ with $X_{1,m}$ varying in a neighborhood. Therefore, $W_1(\epsilon),W_2(\epsilon),W_3(\epsilon)$ only depends on $A,x,y,z_1,z_2,X_{1,m}$. 

\textbf{Step 2.} As $\epsilon\to 0$, we can verify that $\tilde{A}_3^{-1}(\epsilon)/\epsilon\to A_2^{\T}A_1^{-2}A_2+I$, $\tilde{x}_2/\epsilon\to A_2^{\T} A_1^{-2}x_1$, $\tilde{y}_2/\epsilon\to A_2^{\T} A_1^{-2}y_1$ and thus
\begin{equation*}
\begin{split}
&W_1(\epsilon)\to z_1-x_1^\T A_1^{-1}x_1,\\
&W_2(\epsilon)\to X_{1,m}-x_1^\T A_1^{-1}y_1,\\
&W_3(\epsilon)\to z_2-y_1^\T A_1^{-1}y_1.\\
\end{split}
\end{equation*}
By the proof of Lemma \ref{lem::det-quad}, $\mathrm{det}(Y_1)/\mathrm{det}(A_1)$ only depends on $z_1-x_1^\T A_1^{-1}x_1$, $z_2-y_1^\T A_1^{-1}y_1$, $x_1^\T A_1^{-1}y_1$, $X_{1,m}$. Therefore, $\mathrm{det}(Y_1)/\mathrm{det}(A_1)$ only depends on $X_{1,m}$ and the limits of $W_1(\epsilon)$, $W_2(\epsilon)$, $W_3(\epsilon)$, as $\epsilon\to0$.

\end{proof}

\subsection{Proof of Theorem \ref{thm::psd2} and Propositions \ref{prop::pd}--\ref{prop::psd}}
\begin{proof}[Proof of Theorem \ref{thm::psd2}]
Since any principal submatrix of a PSD matrix is a PSD matrix and thus has a nonnegative determinant, we have that (i) implies (ii) and (iii). It is straightforward that (iii) implies (ii). We only need to show that (ii) implies (i). Next, we show (ii) implies (i) by mathematical induction. 

When $m=2$, the result is correct by classic Sylvester's criterion in Theorem \ref{thm::syl}. Suppose that (ii) implies (i) when $m=m_0$. When (ii) holds for $m=m_0+1$, by induction, $X_{1:m_0,1:m_0}$ and $X_{2:(m_0+1),2:(m_0+1)}$ are PSD. We discuss the following two cases: 

\begin{enumerate}
\item \textbf{Case 1}: $X_{2:m_0,2:m_0}$ is invertible. This implies that $X_{2:m_0,2:m_0}$ is PD and the inner-saturated submatrix of $X$ is $X$ itself. By Lemma \ref{lem::psd-suff1}, $X$ is PSD.  

\item \textbf{Case 2}: $X_{2:m_0,2:m_0}$ is not invertible. Suppose $X_{I,I}$ is one inner-saturated submatrix of $X$ satisfying $\mathrm{det}(X_{I,I})\geq 0$, where $I=\{1,J,m_0+1\}$ and $J$ satisfies that $\{X_{2:m_0,j}:j\in J\}$ is a maximal linearly independent set of the column vectors in $X_{2:m_0,2:m_0}$. This implies that $X_{J,J}$ is PSD. By Lemma \ref{lem::psd-suff1}, $X_{I,I}$ is PSD. Next, we show $\mathrm{det}(X_{I',I'})\geq 0$ for any $I'\subseteq\{1,2,...,m_0+1\}$, and thus $X$ is PSD by classic Sylvester's criterion in Theorem \ref{thm::syl}. We discuss the following two subcases: 
\begin{itemize}
\item \textbf{Subcase 1}: $1\not\in I'$ or $(m_0+1)\not\in I'$. Since $X_{I',I'}$ is a principal submatrix of PSD matrix $X_{1:m_0,1:m_0}$ or $X_{2:(m_0+1),2:(m_0+1)}$, we have $\mathrm{det}(X_{I',I'})\geq 0$. 
\item \textbf{Subcase 2}: both $1\in I'$ and $(m_0+1)\in I'$. Let $J'=I'-\{1,m_0+1\}$. There are two possibilities: 
\begin{itemize}
\item If $X_{J',J'}$ is not invertible, then $\mathrm{det}(X_{I',I'})=0$ by Lemma \ref{lem::inner-saturated}. 
\item If $X_{J',J'}$ is invertible, then there must exist an index set $J''$ satisfying that $J'\subseteq J''$ and $\{X_{2:m_0,j}:j\in J''\}$ is a maximal linearly independent set of the column vectors in $X_{2:m_0,2:m_0}$. Let $I''=\{1,J'',m_0+1\}$. By Lemma \ref{lem::inner-saturated}, we have $\mathrm{det}(X_{I'',I''})\geq 0$, Moreover, by Lemma \ref{lem::psd-suff1}, $X_{I'',I''}$ is PSD, which further implies that $\mathrm{det}(X_{I',I'})\geq 0$. 
\end{itemize}
\end{itemize}
\end{enumerate}
Therefore, (ii) implies (i) when $m=m_0+1$. We finish the proof by mathematical induction. 
\end{proof}

\begin{proof}[Proof of Proposition \ref{prop::pd}]
For part (i), if $X$ is PD, then $X$ has a positive determinant by Theorem \ref{thm::syl}(1). Moreover, if $X_{1:(m-1),1:(m-1)}$ is PD and $X$ has a positive determinant, then all of leading principal minors of $X$ are positive, and thus $X$ is PD by Theorem \ref{thm::syl}(1). Therefore, $X$ is PD if and only if $X$ has a positive determinant. 

For part (ii), $\mathrm{det}(X)$ is a quadratic function of $X_{1,m}$ given all other elements in $X$. In particular, $\mathrm{det}(X)=aX_{1,m}^2+bX_{1,m}+c$ where $a,b,c$ satisfy $a<0$ and $b^2-4ac>0$, by Lemma \ref{lem::det-quad}. Therefore, there exists a real number $k$ such that we can set $X_{1,m}=X_{m,1}=k$ to ensure $X$ is PD. 

\end{proof}

\begin{proof}[Proof of Proposition \ref{prop::psd}]
For part (i), if $X$ is PSD, then one of the inner-saturated submatrices of $X$ has a nonnegative determinant by Theorem \ref{thm::psd2}. Moreover, if both $X_{1:(m-1),1:(m-1)}$ and $X_{2:m,2:m}$ are PD and one of the inner-saturated submatrices of $X$ has a nonnegative determinant, then Theorem \ref{thm::psd2}(ii) holds, and thus $X$ is PSD by Theorem \ref{thm::psd2}. Therefore, $X$ is PSD if and only if one of the inner-saturated submatrices of $X$ has a nonnegative determinant. 

For part (ii), we focus on one specific inner-saturated submatrix of $X$, denoted as $Y$. Then, $\mathrm{det}(Y)$ is a quadratic function of $X_{1,m}$ given all other elements in $Y$. In particular, $\mathrm{det}(Y)=aX_{1,m}^2+bX_{1,m}+c$ where $a,b,c$ satisfy $a<0$ and $b^2-4ac\geq 0$, by Lemma \ref{lem::det-quad}. Therefore, there exists a real number $k$ such that we can set $X_{1,m}=X_{m,1}=k$ to ensure $X$ is PSD. 

\end{proof}

\bibliographystyle{Chicago}

\bibliography{sylvester}

\begin{thebibliography}{}

\bibitem[\protect\citeauthoryear{Barrett, Johnson, and Lundquist}{Barrett
  et~al.}{1989}]{barrett1989determinantal}
Barrett, W.~W., C.~R. Johnson, and M.~Lundquist (1989).
\newblock Determinantal formulae for matrix completions associated with chordal
  graphs.
\newblock {\em Linear Algebra and its Applications\/}~{\em 121}, 265--289.

\bibitem[\protect\citeauthoryear{Grone, Johnson, S{\'a}, and Wolkowicz}{Grone
  et~al.}{1984}]{grone1984positive}
Grone, R., C.~R. Johnson, E.~M. S{\'a}, and H.~Wolkowicz (1984).
\newblock Positive definite completions of partial {H}ermitian matrices.
\newblock {\em Linear Algebra and its Applications\/}~{\em 58}, 109--124.

\bibitem[\protect\citeauthoryear{Horn and Johnson}{Horn and
  Johnson}{2012}]{horn2012matrix}
Horn, R.~A. and C.~R. Johnson (2012).
\newblock {\em Matrix Analysis}.
\newblock Cambridge: Cambridge University Press.

\bibitem[\protect\citeauthoryear{Johnson}{Johnson}{1990}]{johnson1990matrix}
Johnson, C.~R. (1990).
\newblock Matrix completion problems: a survey.
\newblock In {\em Matrix theory and applications}, Volume~40, pp.\  171--198.

\bibitem[\protect\citeauthoryear{Kwan}{Kwan}{2021}]{kwan2021remedies}
Kwan, C.~C. (2021).
\newblock Remedies for misapplications of sylvester’s criterion: A pedagogic
  illustration.
\newblock {\em Spreadsheets in Education\/}.

\bibitem[\protect\citeauthoryear{Prussing}{Prussing}{1986}]{prussing1986principal}
Prussing, J.~E. (1986).
\newblock The principal minor test for semidefinite matrices.
\newblock {\em Journal of Guidance, Control, and Dynamics\/}~{\em 9}, 121--122.

\bibitem[\protect\citeauthoryear{Yamashita and Yabe}{Yamashita and
  Yabe}{2015}]{yamashita2015survey}
Yamashita, H. and H.~Yabe (2015).
\newblock A survey of numerical methods for nonlinear semidefinite programming.
\newblock {\em Journal of the Operations Research Society of Japan\/}~{\em 58},
  24--60.

\bibitem[\protect\citeauthoryear{Zhang and Ding}{Zhang and
  Ding}{2022}]{zhang2022interpretable}
Zhang, M. and P.~Ding (2022).
\newblock Interpretable sensitivity analysis for the {B}aron--{K}enny approach
  to mediation with unmeasured confounding.
\newblock {\em arXiv preprint arXiv:2205.08030\/}.

\end{thebibliography}

\end{document}